\documentclass[12pt]{amsart}
\usepackage{amssymb,latexsym}
\usepackage{varioref}
\usepackage{color}
\usepackage{amscd}
\labelformat{equation}{(#1)}

\begin{document}

\title[Borel degenerations of ACM curves]
{Borel Degenerations of Arithmetically Cohen-Macaulay curves in
$\mathbb{P}^3$}

\author{Gunnar Fl{\o}ystad}
\address{Matematisk Institutt\\
         Johs. Brunsgt. 12\\
         5008 Bergen}
\email{gunnar@mi.uib.no}

\author{Margherita Roggero}
\address{Dipartimento di Matematica dell'Universit\'a di Torino\\
         Via Carlo Alberto 10\\
         10123 Torino, Italy}
\email{margherita.roggero@unito.it}

\date{\today}

\begin{abstract}
We investigate Borel ideals on the Hilbert scheme components 
of arithmetically Cohen-Macaulay (ACM) codimension two schemes in 
$\mathbb{P}^n$. 
We give a basic necessary criterion for a Borel ideal to be on
such a component. Then considering ACM curves in $\mathbb{P}^3$ on a quadric
we compute in several examples all the Borel ideals on their Hilbert
scheme component. Based on this we conjecture which Borel ideals are
on such a component, and for a range of Borel ideals we prove that
they are on the component.
\end{abstract}
\maketitle

\theoremstyle{plain}
\newtheorem{theorem}{Theorem}[section]
\newtheorem{corollary}[theorem]{Corollary}
\newtheorem*{main}{Main Theorem}
\newtheorem{lemma}[theorem]{Lemma}
\newtheorem{satlemma}[theorem]{Saturation Lemma}
\newtheorem{baslemma}[theorem]{Basic Limits Lemma}
\newtheorem{proposition}[theorem]{Proposition}
\newtheorem{conjecture}[theorem]{Conjecture}
\newtheorem*{theoremA}{Theorem}
\newtheorem*{theoremB}{Theorem}

\theoremstyle{definition}
\newtheorem{definition}[theorem]{Definition}
\newtheorem{question}[theorem]{Question}
\newtheorem*{claim}{Claim}

\theoremstyle{remark}
\newtheorem{notation}[theorem]{Notation}
\newtheorem{remark}[theorem]{Remark}
\newtheorem{example}[theorem]{Example}

\labelformat{equation}{(#1)}
\def\mathbi#1{\textbf{\em #1}}
\newcommand{\hilbp}{{\mathcal{H}\textnormal{ilb}_{\mathbi{p}(z)}^n}}
\newcommand{\hilb}{{\mathcal{H}\textnormal{ilb}}}
\newcommand{\BSt}{\mathcal{M}f}
\newcommand{\cN}{\mathcal{N}}
\newcommand{\Gr}{Gr\"obner}
\newcommand{\Pl}{Pl\"ucker}
\newcommand{\St}{\mathcal{S}t}
\def\gm{{\mathfrak m}}

\newcommand{\psp}[1]{{{\bf P}^{#1}}}
\newcommand{\psr}[1]{{\bf P}(#1)}
\newcommand{\op}{{\mathcal O}}
\newcommand{\opw}{\op_{\psr{W}}}
\newcommand{\go}{\op}

\newcommand{\ini}{\text{in}}
\newcommand{\gin}[1]{\text{gin}(#1)}
\newcommand{\kr}{{k}}
\newcommand{\kk}{{k}}
\newcommand{\pd}{\partial}
\newcommand{\vardel}{\partial}
\renewcommand{\tt}{{\bf t}}


\newcommand{\coh}{{{\text{{\rm coh}}}}}


\newcommand{\modv}[1]{{#1}\text{-{mod}}}
\newcommand{\modstab}[1]{{#1}-\underline{\text{mod}}}

\newcommand{\sut}{{}^{\tau}}
\newcommand{\sumit}{{}^{-\tau}}
\newcommand{\til}{\thicksim}

\newcommand{\totp}{\text{Tot}^{\prod}}
\newcommand{\dsum}{\bigoplus}
\newcommand{\dprod}{\prod}
\newcommand{\lsum}{\oplus}
\newcommand{\lprod}{\Pi}

\newcommand{\La}{{\Lambda}}
\newcommand{\lam}{{\lambda}}
\newcommand{\GL}{{GL}}

\newcommand{\sirstj}{\circledast}

\newcommand{\she}{\EuScript{S}\text{h}}
\newcommand{\cm}{\EuScript{CM}}
\newcommand{\cmd}{\EuScript{CM}^\dagger}
\newcommand{\cmri}{\EuScript{CM}^\circ}
\newcommand{\cler}{\EuScript{CL}}
\newcommand{\clerd}{\EuScript{CL}^\dagger}
\newcommand{\clerri}{\EuScript{CL}^\circ}
\newcommand{\gor}{\EuScript{G}}
\newcommand{\gF}{\mathcal{F}}
\newcommand{\gG}{\mathcal{G}}
\newcommand{\gM}{\mathcal{M}}
\newcommand{\gE}{\mathcal{E}}
\newcommand{\gD}{\mathcal{D}}
\newcommand{\gI}{\mathcal{I}}
\newcommand{\gP}{\mathcal{P}}
\newcommand{\gK}{\mathcal{K}}
\newcommand{\gL}{\mathcal{L}}
\newcommand{\gS}{\mathcal{S}}
\newcommand{\gC}{\mathcal{C}}
\newcommand{\gO}{\mathcal{O}}
\newcommand{\gJ}{\mathcal{J}}
\newcommand{\gU}{\mathcal{U}}
\newcommand{\mm}{\mathfrak{m}}

\newcommand{\dlim} {\varinjlim}
\newcommand{\ilim} {\varprojlim}

\newcommand{\CM}{\text{CM}}
\newcommand{\Mon}{\text{Mon}}


\newcommand{\Kom}{\text{Kom}}


\newcommand{\EH}{{\mathbf H}}
\newcommand{\res}{\text{res}}
\newcommand{\Hom}{\text{Hom}}
\newcommand{\inhom}{{\underline{\text{Hom}}}}
\newcommand{\Ext}{\text{Ext}}
\newcommand{\Tor}{\text{Tor}}
\newcommand{\ghom}{\mathcal{H}om}
\newcommand{\gext}{\mathcal{E}xt}
\newcommand{\id}{\text{{id}}}
\newcommand{\im}{\text{im}\,}
\newcommand{\codim} {\text{codim}\,}
\newcommand{\resol}{\text{resol}\,}
\newcommand{\rank}{\text{rank}\,}
\newcommand{\lpd}{\text{lpd}\,}
\newcommand{\coker}{\text{coker}\,}
\newcommand{\supp}{\text{supp}\,}
\newcommand{\Ad}{A_\cdot}
\newcommand{\Bd}{B_\cdot}
\newcommand{\Fd}{F_\cdot}
\newcommand{\Gd}{G_\cdot}


\newcommand{\sus}{\subseteq}
\newcommand{\sups}{\supseteq}
\newcommand{\pil}{\rightarrow}
\newcommand{\vpil}{\leftarrow}
\newcommand{\rpil}{\leftarrow}
\newcommand{\lpil}{\longrightarrow}
\newcommand{\inpil}{\hookrightarrow}
\newcommand{\pils}{\twoheadrightarrow}
\newcommand{\projpil}{\dashrightarrow}
\newcommand{\dotpil}{\dashrightarrow}
\newcommand{\adj}[2]{\overset{#1}{\underset{#2}{\rightleftarrows}}}
\newcommand{\mto}[1]{\stackrel{#1}\longrightarrow}
\newcommand{\vmto}[1]{\overset{\tiny{#1}}{\longleftarrow}}
\newcommand{\mtoelm}[1]{\stackrel{#1}\mapsto}

\newcommand{\eqv}{\Leftrightarrow}
\newcommand{\impl}{\Rightarrow}

\newcommand{\iso}{\cong}
\newcommand{\te}{\otimes}
\newcommand{\into}[1]{\hookrightarrow{#1}}
\newcommand{\ekv}{\Leftrightarrow}
\newcommand{\equi}{\simeq}
\newcommand{\isopil}{\overset{\cong}{\lpil}}
\newcommand{\equipil}{\overset{\equi}{\lpil}}
\newcommand{\ispil}{\isopil}
\newcommand{\vvi}{\langle}
\newcommand{\hvi}{\rangle}
\newcommand{\susneq}{\subsetneq}
\newcommand{\sgn}{\text{sign}}
\newcommand{\prikk}{\bullet}


\newcommand{\xd}{\check{x}}
\newcommand{\ortog}{\bot}
\newcommand{\tL}{\tilde{L}}
\newcommand{\tM}{\tilde{M}}
\newcommand{\tH}{\tilde{H}}
\newcommand{\tvH}{\widetilde{H}}
\newcommand{\tvh}{\widetilde{h}}
\newcommand{\tV}{\tilde{V}}
\newcommand{\tS}{\tilde{S}}
\newcommand{\tT}{\tilde{T}}
\newcommand{\tR}{\tilde{R}}
\newcommand{\tf}{\tilde{f}}
\newcommand{\ts}{\tilde{s}}
\newcommand{\tp}{\tilde{p}}
\newcommand{\tr}{\tilde{r}}
\newcommand{\tfst}{\tilde{f}_*}
\newcommand{\empt}{\emptyset}
\newcommand{\bfa}{{\mathbf a}}
\newcommand{\bfb}{{\mathbf b}}
\newcommand{\bfd}{{\mathbf d}}
\newcommand{\bfe}{{\mathbf e}}
\newcommand{\bfp}{{\mathbf p}}
\newcommand{\bfc}{{\mathbf c}}
\newcommand{\bfl}{{\mathbf t}}
\newcommand{\bfo}{{\mathbf 0}}
\newcommand{\bfx}{{\mathbf x}}
\newcommand{\la}{\lambda}
\newcommand{\bfen}{{\mathbf 1}}
\newcommand{\ep}{\epsilon}
\newcommand{\en}{r}
\newcommand{\tu}{s}
\newcommand{\carc}{\mbox{char.}}
\newcommand{\D}{\Delta}
\newcommand{\Dm}{\D^{-}}

\newcommand{\Spec}{Spec\,}
\newcommand{\ome}{\omega_E}
\newcommand{\tI}{\tilde{I}}

\newcommand{\bevis}{{\bf Proof. }}
\newcommand{\demofin}{\qed \vskip 3.5mm}
\newcommand{\nyp}[1]{\noindent {\bf (#1)}}
\newcommand{\demo}{{\it Proof. }}
\newcommand{\demodone}{\demofin}
\newcommand{\parg}{{\vskip 2mm \addtocounter{theorem}{1}  
                   \noindent {\bf \thetheorem .} \hskip 1.5mm }}

\newcommand{\red}{{\text{red}}}
\newcommand{\lcm}{{\text{lcm}}}
\newcommand{\Af}{{\mathbb A}^1}

\newcommand{\ucong}{\triangleq}


\newcommand{\dl}{\Delta}
\newcommand{\cdel}{{C\Delta}}
\newcommand{\cdelp}{{C\Delta^{\prime}}}
\newcommand{\dlst}{\Delta^*}
\newcommand{\Sdl}{{\mathcal S}_{\dl}}
\newcommand{\lk}{\text{lk}}
\newcommand{\lkd}{\lk_\Delta}
\newcommand{\lkp}[2]{\lk_{#1} {#2}}
\newcommand{\del}{\Delta}
\newcommand{\delr}{\Delta_{-R}}
\newcommand{\dd}{{\dim \del}}


\newcommand{\Bo}{J}

\renewcommand{\aa}{{\bf a}}
\newcommand{\bb}{{\bf b}}
\newcommand{\cc}{{\bf c}}
\newcommand{\xx}{{\bf x}}
\newcommand{\yy}{{\bf y}}
\newcommand{\zz}{{\bf z}}
\newcommand{\mv}{{\xx^{\aa_v}}}
\newcommand{\mF}{{\xx^{\aa_F}}}

\newcommand{\pnm}{{\bf P}^{n-1}}
\newcommand{\opnm}{{\go_{\pnm}}}
\newcommand{\ompnm}{\omega_{\pnm}}

\newcommand{\pn}{{\mathbb P}^n}
\newcommand{\ptre}{{\mathbb P}^3}
\newcommand{\hele}{{\mathbb Z}}
\newcommand{\nat}{{\mathbb N}}
\newcommand{\rasj}{{\mathbb Q}}

\newcommand{\dt}{{\displaystyle \cdot}}
\newcommand{\st}{\hskip 0.5mm {}^{\rule{0.4pt}{1.5mm}}}              
\newcommand{\disk}{\scriptscriptstyle{\bullet}}
\newcommand{\sat}{{sat}}

\newcommand{\cF}{F_\dt}
\newcommand{\pol}{f}

\newcommand{\disc}{\circle*{5}}

\def\CC{{\mathbb C}}
\def\GG{{\mathbb G}}
\def\ZZ{{\mathbb Z}}
\def\NN{{\mathbb N}}
\def\RR{{\mathbb R}}
\def\OO{{\mathbb O}}
\def\QQ{{\mathbb Q}}
\def\VV{{\mathbb V}}
\def\PP{{\mathbb P}}
\def\EE{{\mathbb E}}
\def\FF{{\mathbb F}}
\def\AA{{\mathbb A}}

\def\bcr{\color{blue} }
\def\ecr{\color{black} }
\def\bcR{\color{red} }
\def\bcV{\color{green} }

\newcommand{\Mf}{\mathcal{M}\!f}
\newcommand{\Ht}{\textnormal{Ht}}
\newcommand{\Supp}{\textnormal{Supp}}
\newcommand{\RevLex}{\mathtt{RevLex}}


\newcounter{cont}\newenvironment{lista}        
    {\begin{list}  {$\bullet$} {\setlength{\leftmargin}{0pt}    \setlength{\rightmargin}{0pt}     
    \setlength{\itemindent}{0pt}                \setlength{\labelwidth}{20pt}                \usecounter{cont}     
    }}              {\end{list}}\newcounter{contbis}\newenvironment{listabis}            {\begin{list}           
    {\textbf{\roman{contbis})}}               { \setlength{\leftmargin}{40pt}                \setlength{\rightmargin}{0pt}     
    \setlength{\itemindent}{0pt}                \setlength{\labelwidth}{20pt}                \usecounter{contbis}               }}      
    {\end{list}}

\section*{Introduction}

The ideal of any subscheme in a projective space $\pn$ may be 
degenerated through coordinate changes, to a Borel fixed monomial ideal
(henceforth called a Borel ideal). So any component of the Hilbert scheme
of subschemes of $\pn$ contains a Borel ideal.

 Borel ideals in characteristic zero have nice combinatorial descriptions.
Borel ideals are also the most degenerate of all ideals in the
sense that if we degenerate a Borel ideal $J$ to another monomial ideal
$J^\prime$ through coordinate changes, then $J^\prime$ is simply obtained
from $J$ through a permuation of the variables. Put in another way,
the $GL(n+1)$-orbit on the Hilbert scheme of a Borel ideal, is closed.
This raises the problem of investigating and finding these {the} most
degenerate ideals on a component. For instance A.Reeves, \cite{RA},
asks if the set of Borel ideals on a component characterizes the component.

The interest in the geography of Borel ideals on the Hilbert scheme
may be said to date back to Hartshorne's proof of the connectedness of
the Hilbert
scheme. Proceeding through a succession of distractions and degenerations
one may proceed from any Borel ideal to the lex segment Borel ideal.
Surprisingly it was shown that this ideal is a smooth point on the
Hilbert scheme, \cite{RS}, thus identifying a distinguished component, 
the lex segment component of the Hilbert scheme. P.Lella in \cite{Le}
shows how Borel ideals may be connected by irreducible rational 
curves on the Hilbert scheme and so provides insight into the network
of Borel ideals on the Hilbert scheme, and in particular provides
a new proof of its connectedness.

\medskip
In this paper we consider the Hilbert scheme components of arithmetically
Cohen-Macaulay (ACM) subschemes of codimension two in $\pn$. These are 
characterized by their homogeneous ideal in  $S:=\kk[x_0, \ldots, x_n]$\ 
having the shortest possible resolution by free $S$-modules, of length two.
The Hilbert scheme components of ACM codimension two subschemes are well
classified. In particular there is a one-to-one correspondence between
such components and ACM Borel fixed ideals of codimension two for the
ordering $x_0 > x_{1} > \cdots$. They have the following form
\[ J(a;\bfb) =(x_0^a, x_0^{a-1}x_{1}^{b_1}, x_0^{a-2}x_{1}^{b_2}, 
 \ldots, x_{1}^{b_a}) \]
where $0 = b_0 < b_1 < \cdots < b_a$. 
 
We find two basic necessary conditions for a Borel ideal to be on the
component of  the above ideal. The first condition is:

\medskip
\noindent {\bf Theorem \ref{NecTheMain}.}
{\it Let  $J$ be a Borel ideal on the Hilbert scheme component of $J(a;\bfb)$.
Let $d_s = \sum_{i = 1}^s b_i$. Then $x_0^{a-s}x_{1}^{d_s}$ is in $J$ for
each $s = 0, \ldots, a$. }
\medskip

The second condition is standard and 
follows by the semi-continuity of the cohomology
of coherent sheaves. 

\medskip
\noindent {\bf Condition 2.} If $J$ is a saturated Borel ideal on there is component
of $J(a;\bfb)$, then there is the inequality of  Hilbert functions
 $h_{S/J(d)} \leq h_{S/J(a;\bfb)}(d)$) for all $d$ .

\medskip
We then proceed to investigate closer what are the Borel ideals on
specific components. This is a hard task and to obtain reasonably 
comprehensive results we restrict ourselves to the case of ACM
curves on quadrics in $\ptre$. The components of such curves
correspond to Borel ideals
\[ J(l,m) = (x^2, xy^l, y^{l+m}) \]
where $l,m \geq 1$ {and $S=k[x,y,z,w]$}. 
In several example cases, for the following values of $(l,m)$:
\[ (1,3),\, (2,2),\, (3,1), \,(3,3), \]
we find by computation all Borel ideals on this component.
For instance when $(l,m) = (3,3)$ there are $989$ Borel ideals with the 
same Hilbert polynomial as $J(3,3)$, but only $45$ of these are on the 
component of $J(3,3)$. 

 In all the computed cases the only obstructions we have found for a
Borel ideal to be on the component of $J(l,m)$ are given by 
Theorem \ref{NecTheMain} and Condition 2. We therefore make
the following.

\medskip
\noindent {\bf {Conjecture \ref{ExConjObs}.}}
{\it 
If a saturated Borel ideal has the same Hilbert polynomial
as $J(l,m)$, then it is on the component of $J(l,m)$ 
if and only if it fulfills
the criteria of Theorem \ref{NecTheMain} and Condition 2.
}

\medskip
We then exhibit many classes of Borel ideals that are on the component
of $J(l,m)$. To do this we consider explicit families of curves defined
by the $2 \times 2$-minors of the matrix
\[ \left[  \begin{matrix} x & y^l & -F \\
                          0 & x   & y^m
           \end{matrix}  \right ].
\]

By specializing $F$ in various ways we get various Borel ideals
as specializations. In Sections 3, 4, and 5, we investigate
in particular 
the cases when $l = 1, 2$, and $3$ and show many classes of Borel ideals
to be on the component of $J(l,m)$.
In the last Section 6 we give a class of Borel
ideals for general $l$ which is on the component.

\medskip
The organization of the paper is as follows. In Section 1 we recall basic
facts about Hilbert scheme components containing ACM codimension two 
subschemes of $\pn$ and prove the basic condition Theorem \ref{NecTheMain}
on Borel ideals on such a component.
In Section 2 we study example cases and compute all the Borel ideals
on the component of $J(l,m)$ in the range of $(l,m)$ stated above.
We also conjecture what are the Borel ideals on the component of $J(l,m)$.
In Section 3 we state our main theorems of sufficient conditions for 
a Borel ideal to be on the component of $J(l,m)$. In Section 4
we give the families of ideals that we degenerate, and in Section 5 we prove
the results. In Section 6 we exhibit a general class of Borel ideals
on the component of $J(l,m)$.

\section{A necessary condition on Borel degenerations}

In this section we consider components of the Hilbert scheme
whose general point corresponds to arithmetically Cohen-Macaulay (ACM)
schemes of codimension two. These components are well classified. We
give a necessary condition for a Borel ideal to be on such a component.

\subsection{ACM codimension two components of the Hilbert scheme} 
A subscheme $X \sus \pn_k$ where $\kk$ is a field, 
is {\it arithmetically Cohen-Macaulay (ACM)}
if its saturated homogeneous ideal $I = I_X$ in the polynomial ring
$S = \kk[x_0, \ldots, x_n]$ has a minimal free resolution of 
length two
\begin{equation*} \label{NecLigRes} 
I \vpil G_0 \vmto{\phi} G_1
\end{equation*}
where $G_0$ and $G_1$ are graded free $S$-modules.

Let $H$ be the Hilbert scheme corresponding to the Hilbert polynomial
of the quotient ring of $I$.
There is a universal family of schemes
\[ \begin{matrix} Z & \sus & H \times \pn \\
                    &      & \downarrow  \\
                    &      & H
   \end{matrix} \]
flat over $H$, and let $\gI_Z$ be its ideal sheaf in $\gO_{H \times \pn}$.
If we have a morphism from an affine scheme $\Spec B  \pil H$ 
we may pull back $\gI_Z$ and get an ideal sheaf $\gI_B$ in $\pn_B$.

Denote by $I_B$
the graded ideal in $B[x_0, \ldots, x_n]$ of global sections
$\oplus_{d \in \hele} \Gamma(\pn_B, \gI_B(d))$. Note that since
$\Spec B$ is affine, the sheafification of $(I_B)_d$ over $\Spec B$
is the pushdown $p_*\gI_B(d)$ by the natural map $p: \pn_B \pil \Spec B$.
The graded ideal $I_B$ will in general not be a flat family of ideals over
$\Spec B$. We shall however see that in our situation, there is an open
subset $\Spec B$ of the Hilbert scheme $H$ such that $I_{B}$ becomes a 
flat family of 
ACM codimension two ideals with the same resolution as $I$. 

\bigskip

\begin{proposition} \label{NecProJB}
Let $n \geq 3$ and 
\[ I \vpil G_0 \vpil G_1 \]
be the minimal free resolution of an ACM ideal $I$ of 
codimension two, 
corresponding to a point {$\mathfrak i$} on the Hilbert scheme $H$. 
Then there is an open affine subset $U = \Spec B \sus H$
containing {$\mathfrak i$} such that the ideal of graded global sections $I_B$ defined
above is flat over $B$ and has a resolution
\[ I_B \vpil G_0 \te_{\kk} B \vmto{\phi_B} G_1 \te_{\kk} B \]
whose fibre at the point {$\mathfrak i$} 
is the resolution of $I$.
\end{proposition}

\begin{proof}
Let $\Spec B \sus H$ 
be an open affine neighbourhood of {$\mathfrak i$}.
We get an ideal sheaf $\gI_B$ in
$\pn_B$. 
For each $i = 1, \ldots, n$ there is a $d(i)$ such that $R^ip_* \gI_B(d)$ 
vanishes for $d \geq d(i)$, by \cite[Thm. III.5.2]{Ha}.
Let $d_0$ be the maximum of the $d(i)$. 

By the Cohomology and { Base} Change Theorem (CBCT)  \cite[Thm. III.12.11]{Ha} 
part b. we get that $p_* \gI_B(d)$ is a locally free (or flat) $B$-module
for $d \geq d_0$,
and that all the cohomology modules $H^i(\pn, \gI_{k({\mathfrak b})}(d))$ vanish
for ${\mathfrak b} \in \Spec B$ when $i > 0$ and $d \geq d_0$.

Also note that $p_* \gI_B(d)$ and all the cohomology modules
$H^0(\pn_{k(\mathfrak b)}, \gI_{k({\mathfrak b})}(d))$ vanish for $d < 0$ since they 
are submodules of $p_* \gO_{\pn_B} (d)$ and 
$\kk({\mathfrak b})[x_0, \ldots, x_n ]_d$ 
respectively.

The fibre ideal $\gI_{k({\mathfrak i})}$ has $I$ as its associated graded ideal.
Since $n \geq 3$, by running the long exact cohomology sequence on 
the sheafification of the resolution of $I$, we get the vanishing of 
$H^1(\pn_{k({\mathfrak i})}, \gI_{k({\mathfrak i})}(d))$ for all $d$.
By semi-continuity of cohomology, there is an open subset
$\Spec B_f$ of $\Spec B$, containing ${\mathfrak i}$, such that the 
$H^1(\pn_{k({\mathfrak b})}, \gI_{k({\mathfrak b})}(d))$ 
vanish
for ${\mathfrak b}$ in $\Spec B_f$ and $d = 0, \ldots, d_0$. But then
we know by the above that they also vanish for all $d \geq 0$. 

{By part a. of CBCT  }
the maps
\[ R^1 p_*(\gI_{B_f}(d)) \te_B k({\mathfrak b})
\pil H^1(\pn_{k({\mathfrak b})}, \gI_{k({\mathfrak b})}(d)), \]
being surjective, will be isomorphisms for ${\mathfrak b}$ in $\Spec B_f$. 
Hence by {Nakayama's} lemma, we obtain the vanishing of $R^1 p_*(\gI_{B_f}(d))$
for all $d \geq 0$.

By part b. of CBCT (applied when $i = 1$) the maps
\[ R^0 p_*(\gI_{B_f}(d)) \te_{B_f} k({\mathfrak b}) \pil 
H^0(\pn_{k({\mathfrak b})},\gI_{k({\mathfrak b})}) \]
are also surjective for $d \geq 0$ and ${\mathfrak b}$ in $\Spec B_f$.
Applying CBCT again (when $i =0$), these maps   
are isomorphisms, and $R^0 p_* (\gI_{B_f}(d))$ is locally free
in a neighbourhood of any such ${\mathfrak b}$, and so on all of $\Spec B_f$. 
Hence the graded ideal $I_{B_{f}}$ is a
flat $B_{f}$-module.

\medskip
For ease of notation denote $B_{f}$ further on simply as $B$. 
We may lift the start of the resolution of $I$
to a diagram
\[ \begin{CD}
I_B @<{p_B}<< G_0 \te_{\kk} B \\
@VVV @VVV \\
I @<{p}<< G_0.
\end{CD} \]
The cokernel of $p_B$ vanishes in the fibre at ${\mathfrak i} 
\in \Spec B$. Since we are in a noetherian setting, $\coker p_B$
has a finite set of generators over $B[x_0, \ldots, x_n]$.
By taking a suitable localization $B_f$ (by abuse of notation we 
still denote it by $B$), all these generators 
vanish and so $p_B$
is surjective. Since $I_B$ and $G_0 \te_{\kk} B$ are $B$-flat, the
kernel of $p_B$ will be $B$-flat, {hence } 
\[ 0 \vpil I_B \te_B k({\mathfrak i})  \vpil (G_0 \te_{\kk} B) \te_B 
k({\mathfrak i})
\vpil (\ker p_B \te_B k({\mathfrak i}))  \vpil 0\] is
exact and so the right term in this sequence is equal to $\ker p$. 
We may now continue the process and lift to a diagram
\[ \begin{CD}
I_B @<{p_B}<< G_0 \te_{\kk} B @<{\phi_B}<< G_1 \te_{\kk} B\\
@VVV @VVV @VVV\\
I @<{p}<< G_0 @<{\phi}<< G_1
\end{CD} \]
and by using flatness we see that the upper
sequence is exact after localizing $B$ suitably.

\end{proof}

\begin{corollary} \label{NecCorACMcomp}
If a component of the Hilbert scheme contains an ACM ideal of 
codimension two, the general point on the component will be an 
ACM codimension two ideal with the same Hilbert function.
\end{corollary}

Such a component will be called { an} ACM codimension two 
component of the Hilbert scheme.

\subsection{Borel ideals} We assume in the following that 
our field $\kk$ has characteristic zero.

A monomial ideal is called a strongly stable ideal
if whenever a monomial $x_jm \in J$ and $i < j$, then the monomial
$x_im \in J$. As $ch(\kk)=0$ this is equivalent to $J$ being  a Borel ideal, namely  invariant for the Borel
subgroup  of $GL(n+1)$ consisting of the upper triangular matrices
(when the linear forms have the ordered  basis $(x_0, x_1, \ldots, x_n)$).
The Borel ideals which are ACM of codimension two are easy to describe.
They are given by their minimal generators as
\begin{equation}  \label{NecLigJ}
 J(a,\bfb) = (x_0^a, x_0^{a-1}x_{1}^{b_1}, x_0^{a-2}x_{1}^{b_2}, 
 \ldots, x_{1}^{b_a})
\end{equation}
where $0 = b_0 < b_1 < \cdots < b_a$. 
Note that  $\frac{1}{(n-2)!}(b_1+\dots +b_a)$ is the leading term of the Hilbert polynomial of $P/J(a,\bf b)$.

If $I$ is any  homogeneous ideal, and we take its generic initial ideal
$\gin{I}$ for the revlex order, then $\gin{I}$ will i) 
be a Borel ideal \cite[15.9]{Eisenbud}, and ii) have the same
depth as $I$, by \cite{BaSt}. Hence if $I$ is ACM of codimension two,
its generic initial ideal will be a Borel ACM codimension two ideal and so
have the form \ref{NecLigJ} above. 
Moreover, $I$ is saturated and has the  same Hilbert function than $J(a,\bfb)$.

Let us now collect the following facts.
\begin{itemize}
\item[1.] Each ACM codimension two component contains a Borel ideal 
of the type $J(a,\bfb)$
(by the above argument).
\item[2.] Such a Borel ideal is a smooth point {on the Hilbert scheme}, 
\cite{El}, and
hence is on a single component.
\item[3.] Distinct ideals $J(a,\bfb)$ in \ref{NecLigJ} 
are on distinct components. This follows by the above
Corollary \ref{NecCorACMcomp}
%
since it is easy to see that distinct pairs $(a,\bfb)$ will 
give distinct
Hilbert functions.
\end{itemize}

In conclusion we get the following well known fact.
\begin{proposition}
There is a one-to-one correspondence between ACM codimension two
components of the Hilbert scheme and ideals $J(a,\bfb)$.
\end{proposition}

Now we shall investigate Borel fixed ideals on the component of 
$J(a,\bfb)$. Let us start with an example.

\begin{example}
Twisted cubic curves in $\ptre$ are ACM curves with Hilbert polynomial 
$3d+1$. The corresponding Borel ideal is 
\[ J = (x_0^2, x_0x_1, x_1^2). \]
It is not difficult to show that the Borel ideal 
{ \[ I = (x_0^2, x_0x_1, x_0x_2, x_1^3) \]} is on the component. 
The ideal
\[ K = (x_0, x_1^3x_2, x_1^4) \]
is the lex segment ideal with Hilbert polynomial $3d+1$. It is a smooth
point {on  the Hilbert scheme and the single component containing it }is
different from the one of $J$.

Another way to see this is the general theorem below which implies that
if a Borel ideal is a degeneration of the ideal of a twisted cubic curve
then it must contain $x_1^3$. 
\end{example}

\begin{theorem} \label{NecTheMain}
Let  $J$ be a Borel ideal on the Hilbert scheme component of $J(a,\bfb)$.
Let $d_s = \sum_{i = 1}^s b_i$. Then $x_0^{a-s}x_{1}^{d_s}$ is in $J$ for
each $s = 0, \ldots, a$. 
\end{theorem}
\begin{proof} We apply Proposition \ref{NecProJB} to the ideal 
$I = J(a,\bfb)$. 
By the Hilbert-Burch theorem the ideal {$I_B$} of Proposition
\ref{NecProJB} is generated by the minors of the matrix $\phi_B$. 
Denote these minors as
\[ {F}_0, {F}_1, \ldots {F}_a. \]
Let $A$ be the local ring at the point ${\mathfrak i}$ in $\Spec B$ 
corresponding to 
$I$.  Considering the $F_i$'s over this local ring we may 
write
\[{F}_i = x_0^{a-i}x_{1}^{b_i} + \sum_j c_{i,j}M_{i,j} \]
where the $M_{i,j}$ are monomials in $\kk[x_0, \ldots, x_n]$ of
degree $a-i + b_i$ and the $c_{i,j}$ are in the maximal ideal
of $A$.
By subtracting multiples of ${F}_0$ we may assume that 
$x_0^a$ is the highest power of $x_0$ occuring in any of the
${F}_i$. Then there will be an open subset $\Spec B_f 
\sus \Spec B$ such that considering the $F_i$ over $B_f$ 
we may for each $i$ write
\[ F_i = x_0^{a-i}x_{1}^{b_i} + \sum_{j=0}^{a} x_0^{a-j}E_{i,j} \]
where the $E_{i,j}$ are polynomials in the variables 
$x_1, \ldots, x_{n}$.
We can then write the transpose 
\[ [F_0, \ldots, F_a]^t = W \cdot [x_0^a, \ldots, x_0, 1]^t \]
where $W$ is an $(a+1) \times (a+1)$ matrix with entries
$E_{i,j}$ in position $(i,j)$ when $i\neq j$ and $x_{1}^{b_i} + E_{i,i}$ 
when $i = j$. For simplicity of notation we denote $B_f$ further on
by $B$. Note that when localizing at ${\mathfrak i}$, all coefficients of
the $F_i$, save the first, are in the maximal ideal of $A$. 

Let $W_s$ be the upper left $(s+1) \times (s+1)$ submatrix of $W$.
Modulo the $B[x_1, \ldots, x_{n}]$-submodule 
$\oplus_{i = 0}^{a-s-1}  x_0^{i}B[x_1, \ldots, x_{n}]$ we get
\[ [F_0, \ldots, F_s]^t \equiv W_s \cdot [x_0^a, \ldots, x_0^{a-s}]^t.  \]
By considering the diagonal of $W_s$ we see that its determinant 
has degree $d_s = \sum_{i = 1}^s b_i$. Also when localizing
at ${\mathfrak i}$ all coefficients of entries on the diagonal are
units of $A$, while those off the diagonal are in the maximal ideal of $A$.
Hence $\det W_s$ is nonzero.  
Let $V_s$ be the matrix of cofactors of $W_s$, so that $V_s \cdot W_s = 
(\det W_s)\cdot I$. 
Then
\[ V_s \cdot [F_0, \ldots, F_s]^t \equiv 
(\det W_s)\cdot  [x_0^a, \ldots, x_0^{a-s}]^t. \]
Hence the last entries in these products are 
\[ G_s := \sum_{i = 0}^s (V_s)_{s,i} F_i \equiv x_0^{a-s} \cdot (\det W_s).\]
Note that the right side has degree $a-s + d_s$. 

\medskip
Now any saturated monomial ideal  
$J$ corresponding to a point in the closure of the open subset $U = \Spec B$
must contain some monomial in $G_s$. 
This will be a monomial of degree $a-s + d_s$ and with $x_0$-degree
$\leq a-s$. So if $J$ is Borel fixed for the ordering
$x_0 > x_{1} > \cdots > x_n$ of the variables, 
it must then contain $x_0^{a-s}x_{1}^{d_s}$. 
\end{proof}

\begin{corollary}
If $a\geq 2$ and $J$ is a Borel ideal that belongs to the component of  
$J(a,\bf b)$, then $x_0\notin J$.
\end{corollary}
\begin{proof}
Let us assume that $ J$ is a strongly stable ideal corresponding to a point of 
the ACM component  of $J(a,\bf b)$. If $x_0 \in J$, then, by the previous 
result  $J$ would contain the ideal $J':= (x_0, x_1^{d_a})$, which defines 
a scheme with  the same codimension  and  degree  as $J(a,\bf b)$. 
But every ideal properly containing  $J'$ defines a scheme with degree lower than 
$d_a$. Hence $J$ cannot include $x_0$.
\end{proof}

\section{Examples and conjectures}

We now consider ACM curves in $\ptre$ and their Hilbert scheme components. 
This section will systematically
investigate example cases where the curves on such a component 
is on a quadric and find all Borel ideals on such a component.

\subsection{ACM ideals on a quadric}
   Since now $n = 3$  the polynomial ring $S$ is   $ \kk[x,y,z,w]$. 
Ordering the variables as $x > y > z > w$, an ACM Borel ideal on a quadric
may then be written as 
\[ \Bo(l,m) = (x^2, xy^l, y^{l+m}). \]
Denote its Hilbert scheme component as $H(l,m)$.
The resolution of $\Bo(l,m)$ is 
\begin{eqnarray*}
 \Bo(l,m) & \xleftarrow{\left [ \begin{matrix} x^2, - xy^l, y^{l+m}
                   \end{matrix} \right ]} & 
S(-2) \oplus S(-l-1) \oplus S(-l-m) \\
&  \xleftarrow{\left [ \begin{matrix} y^l & 0 \\
                              x & y^m \\
                              0   & x
                   \end{matrix} \right ]} & 
S(-l-2) \oplus S(-l-m-1). 
\end{eqnarray*}

\begin{example} The case $l = 1$ and $m = 3$. The Hilbert polynomial is
$5t-2$. There are $7$ (saturated) Borel ideals on the Hilbert scheme
$\hilb^3_{5t-2}$ and these are:
\begin{itemize}
	\item $J_1 = (x, y^6, y^5 z^3)$
	\item $J_2 = (x, y^7, y^6 z, y^5 z^2)$
	\item $J_3 = (x^2, x y, x z, y^6, y^5 z^2)$
	\item $J_4 = (x^2, x y, x z^2, y^6, y^5 z)$
	\item $J_5 = (x^2, x y, x z^3, y^5)$
	\item $J_6 = (x^2, x y^2, x y z, x z^2, y^5) $
	\item $J_7 = (x^2, x y, y^4).$
\end{itemize}

Which of these Borel ideals are on the component $H(1,3)$ of the ACM Borel ideal
$J_7$? By Theorem \ref{NecTheMain} a necessary condition is 
that $xy$ and
$y^5$ are in the Borel ideal. 
This leaves only the possibility of $J_5$ (and of course
$J_7$) to be on $H(1,3)$.
By the Hartshorne result on the connectedness of the Hilbert scheme and the fact that every intersection of components 
contains at least a Borel ideal, we can conclude that at least another Borel ideal in on  $H(1,3)$. Hence, $J_7\in H(1,3)$.
We can also obtain an explicit deformation 
from a smooth point of $H(1,3)$ to $J_5$,   
by letting $F = z^3$ in \ref{ExLigTominor}.  We get the ideal 
$I = (x^2, xy, y^4 + xz^3)$ that  corresponds to a smooth point on $H(1,3)$. 
The initial ideal of $I$ with respect to the lexicographic order is 
$J_5$ (and with respect to the reverse lexicographic order {it} is $J_7$). 
 In conclusion we have
established that the only Borel ideals on $H(1,3)$ are $J_5$ and $J_7$.
\end{example}

\subsection{Conjectures}
 In the above example there were two obstructions which ruled out a saturated
Borel ideal from being on the ACM component. By Theorem \ref{NecTheMain}
we must have:

\begin{itemize}
\item[C1.] If a (saturated) Borel ideal $J$ is on the component $H(l,m)$
then $x^2$, $xy^l$, $y^{2l+m} \in J$. 
\end{itemize}

By semi-continuity of the cohomology of coherent sheaves \cite[III.12]{Ha} 
we must have:

\begin{itemize}
\item[C2.] If a (saturated) Borel ideal $J$ is on the component $H(l,m)$
then the hilbert function   $h_{S/J}(d) \leq h_{S(J(l,m))}(d)$ for all $d \geq 1$. . 
\end{itemize}

In all the examples we have computed,
these are the only two obstructions we have
found.
 We therefore make the following.
\begin{conjecture} \label{ExConjObs}
If a saturated Borel ideal $J$ has the same Hilbert polynomial
as $\Bo(l,m) = (x^2, xy^l, y^{l+m})$, 
then it is on the Hilbert scheme component of $\Bo(l,m)$
if and only if i) $x^2$, $xy^l,  y^{2l+m} \in J$  and 
ii) the Hilbert function
 $h_{S/J}(d) \leq h_{S(J(l,m))}(d)$ for all $d \geq 1$. 
\end{conjecture}

These two conditions  are independent, as the following example shows.
\begin{example} Let us consider the ACM component of the Hilbert scheme with 
Hilbert polynomial $7t-5$ that contains $J(3,1)=(x^2, xy^3, y^4)$. 
All the Borel ideals on this component are listed in  \ref{ExExSyv}.
 
 In particular,  the ideal $J_{90}=(x^2,x y^3,xy^2 z^2,xyz^3,x z^4,y^7)$ has 
Hilbert polynomial $7t-5$ and does not belong to that component, 
 because it satisfies C1, but does not satisfy C2. In fact  in degree 4 we have 
 $h_{S/J_{90}}(4)=26 > h_{S/J(3,1)}(4)=25$. 
 
 On the other hand, there are many ideals with this Hilbert polynomial that satisfy 
C2, but not C1, for instance the Lex segment $J_1=(x,y^8,y^7 z^9)$ 
 (containing $x$) or
 $J_{98} = (x^2,xy^2,xyz,y^8,y^7 z,y^6 z^2)$ (not containing $x$).  
\end{example}

Now we proceed to consider more examples. As soon as we have eliminated all
Borel ideals not fulfilling C1. and C2., the challenge is to show that the 
remaining ideals are on the component $H(l,m)$. We shall do this in several
examples, illustrating computational arguments and techniques. But first
we recall some basic facts on monomial orders.


\subsection{Monomial orders and segment ideals}

 Given integers $v_1 > v_2 > v_3 > v_4$,
a term order $\prec_{[v_1, v_2, v_3, v_4]}$ 
may be specified by the matrix (see \cite[Sec. 1.4]{KrRo}):
\[M = \left(\begin{array}{cccc} 1&1&1&1 \\ v_1& v_2& v_3 & v_4 
\\ 0&0&-1&0 \\ 0&-1& 0&0\end{array}\right).\]
This means that two monomials $\bfx^{\bfa} > \bfx^{\bfb}$ if $M \cdot \bfa 
> M \cdot \bfb$
in the lex order on four-tuples.

If $J$ is a monomial ideal in $S = \kk[x_0, \ldots, x_n]$, 
its graded piece $J_t$ is a {\it segment}
if there is a term order $\prec$ such that the $d = \dim_\kk J_t$
monomials in $J_t$ are the first $d$ monomials in $S_t$ for
this term order. It is easy to see, \cite[Lem. 3.2]{CLMR}
that if $I_t$ is a segment for $\prec$, then $I_s$ is also 
a segment for $\prec$ when $s < t$. 

Two cases are particularly noteworthy, \cite{CLMR}. One case is when $t$ is the
regularity of $J$ (see \cite[20.5]{Eisenbud} for the definition of regularity). 
For a Borel ideal this is simply the degree
of the largest generator of $J$. In this case $J$ is called a reg-segment
ideal. The other case is when $t$ is the 
Gotzmann number $r$ of $J$. This number
depends only on the Hilbert polynomial $p(t)$
of $J$ and it is the largest  regularity 
an ideal with Hilbert polynomial $p(t)$  can have. 
It is the regularity of the
lex segment ideal with this Hilbert polynomial. In this case $J$ is
called a Hilbert segment ideal.

In the following examples, when we have a Borel ideal $J$
fulfilling the necessary conditions C1. and C2. for it to be on $H(l,m)$, 
we check if it is a reg-segment ideal for some order $\prec_{[v_1, v_2, v_3, v_4]}$.
This can be done by the applet Segment available at \cite{LeBo}.
If it is so, take an ACM ideal $I$ corresponding to a point on 
the component $H(l,m)$. Usually an ideal of the form \ref{ExLigTominor}
below, where $F$ is general, or even a completely 
general ideal on $H(l,m)$ (this will be an ACM ideal by Theorem
\ref{NecProJB}).
It corresponds to a smooth
point on $H(l,m)$ by \cite{El}. 
Hence $H(l,m)$ is the unique component containing
the ideal $I$.  We compute
its initial ideal $\text{in}_\prec(I)$ and see if the saturation of this ideal
is $J$. Then $J$ will correspond to a point on $H(l,m)$.
In all cases when $J$ is reg-segment, this approach has worked for us to
show that $J$ is on the component.

\begin{example} The case $l=2$ and $m=2$. The Hilbert polynomial is $6t-3$. 
By using the applet {\texttt BorelGenerator} \cite{LeBo} 
we may get a list of all Borel
ideals with this Hilbert polynomial, and there are $31$ such.
Only seven of them fulfill the conditions C1. and C2. These are (as numbered
by the applet of loc.cit.):
\begin{itemize}
\item $J_{21}= (x^2, xy, y^6, xz^6)$,    
\item $J_{22}= (x^2, xy^2, xyz, y^6, xz^5)$,      
\item $J_{23}= (x^2, xy^2, xyz^2, xz^4, y^6)$,
\item $J_{27}= (x^2, xy, y^6, y^5z^2) $,
\item $J_{28}= (x^2, xy^2, xyz, y^6, y^5z)$,
\item $J_{29}=( x^2, xy^2, xyz^2, y^5)$,
\item $J_{31}= (x^2, xy^2, y^4)$.                          
\end{itemize}

All of the above ideals are reg-segment ideals for various term orders.
A general curve on this ACM component will be a complete intersection 
of a quadric $Q$ and a cubic $C$. Let $I = (Q,C)$ where the two forms
are chosen generic (i.e. randomly). 

By considering the term orders for which the ideals above are segments, 
we find the following initial ideals:
\noindent\begin{itemize} 
\item  $\ini(I,\prec_{[42, 8, 1, 0]})=(x^2, x y^2, x y z^2, xyzw^2 , y^6, 
xyw^4 , x z^6)$ whose saturation is $J_{21}$;
\item  $\ini(I,\prec_{[17, 4, 1, 0]})=(x^2, x y^2, x y z^2, xyzw^2, x z^5, y^6)$ 
whose  saturation is $J_{22}$;
\item   $\ini(I,\prec_{[16, 4, 2, 0]})=J_{23}$;
\item  $\ini(I,\prec_{[50, 12, 1, 0]})=(x^2, x y^2, x y z^2, xyz w^2 , 
xyw^4 , y^6, y^5 z^2)$ whose  saturation is $J_{27}$;
\item    $\ini(I,\prec_{[44, 11, 1, 0]})=(x^2, x y^2, x y z^2, xyzw^2 , y^5z, y^6)$
whose saturation is $J_{28}$;
\item    $\ini(I,\prec_{[ 37, 10, 1,0]})=J_{29}$.
\end{itemize}
\end{example}

Hence all of them are on the component $H(2,2)$.
In the above example  we obtain the same initial ideals also choosing as $I$ 
an ACM ideal on the component $H(l,m)$ 
 which is generated by the $2 \times 2$ minors  
of a matrix of the following type:
\begin{equation} \label{ExLigMatrise} A(F) = 
 \left [ \begin{matrix} x & y^m & -F \\
                          0  & x & y^l
           \end{matrix}   \right ]  
\end{equation}
where $F$ is a polynomial in $\kk[ y, z, w]$, homogeneous
of degree $l+m-1$.
Performing row and column operations
on the matrix, we may  assume that  
\begin{equation} \label{ExLigFform}
F = y^{l-1}F_{m} + y^{l-2} F_{m+1} + \cdots + y^{0}F_{m+l-1}. 
\end{equation}
with $F_i \in \kk[z,w]$.
The $2 \times 2$-minors of the matrix are 
\begin{equation} \label{ExLigTominor} x^2,\quad  x y^l, \quad
 G := xF + y^{m+l}. 
\end{equation}
By multiplying $G$ with $y^l$ we see that the ideal generated by these
minors will also contain $y^{2l+m}$.

Two notable features of the above example are.
\begin{itemize}
\item[1.] Each Borel ideal is the limit of ideals which are generated by the
$2\times 2$-minors of the matrix $A(F)$. In particular
all these ideals contain
$x^2, xy^l$ and $y^{2l+m}$.
\item[2.] The ideal that we degenerate is obtained by a general choice
(either $Q$ and $C$, or $F$).
\end{itemize}

We shall see that in all our examples we are able to do as in 1.
However we are not always able to do as in 2.

\begin{conjecture}
Given a saturated Borel ideal $J$ on the Hilbert scheme component of the
Borel ideal
$\Bo(l,m) = (x^2, xy^l, y^{l+m})$.
Then there is a family of ideals generated by the $2\times 2$-minors of
matrices $A(F)$ of \ref{ExLigMatrise}, 
which specialize to a monomial ideal whose
saturation is $J$. 
\end{conjecture}

\begin{example} \label{ExExSyv}
The case $l = 3$ and $m=1$. This gives the Hilbert polynomial
$p(t) = 7t-5$.
Using the applet {\texttt BorelGenerator} in \cite{LeBo}, 
there are $112$ {saturated} Borel ideals
with this Hilbert polynomial. Of them, $18$ fulfill conditions C1. and C2.
The labels of these $18$ given by {\texttt BorelGenerator} are
\begin{equation} \label{ExLigListe}
 78, 79, 80, 82, 83, 85, 86,  95, 97, 99, 101, 102, 104, 105, 109, 110, 112. 
\end{equation}
All of these, except $J_{83}$, $J_{85}$, $J_{102}$,   are reg-segment ideals. 
We can verify that all these ideals indeed belong to $H(3,1)$, by
computing the initial ideal of a general ACM ideal on $H(3,1)$, or a
general ACM ideal of the type \ref{ExLigTominor}, w.r.t.  
suitable term orderings. 

 Only the following three cases are not segment ideals, as may 
be verified by the simple criterion {\cite[Prop. 3.5]{CLMR}}.
\begin{itemize}
\item $J_{83} = (x^2, xy^3, xy^2z, xyz^2, xz^6, y^7)$,
\item $J_{85} = (x^2, xy^2, xyz^4, xz^5, y^7)$,
\item $J_{102} = (x^2, xy^3, xy^2z, xyz^2, y^6z, y^7)$,
\end{itemize} 
  
For instance in the case of $J_{85}$ one can see that it is not a segment in degree $7$, since
 $ xyz^3w^2 \not \in J_{85}$,
 $xy^2zw^3,xz^5w \in J_{85}$ and  $(xyz^3w^2)^2 = xy^2zw^3 \cdot xz^5w$. Then there cannot exist a term order $\prec $ such that  
 $xyz^3w^2\prec xy^2zw^3$ and $xyz^3w\prec xz^5$. 
 
 
However,  it is possible to verify that these three ideals 
are also on the component.

\medskip
\noindent $\bullet$ Let $I_{83}$ be the ideal 
generated by the $2 \times 2$-minors of $A(F)$ where
\[ F =y^2z+ wzy+2yz^2-w^2z+4z^3. \] 
The initial ideal of $I_{83}$ with respect to the lex order is 
a monomial ideal whose saturation is $J_{83}$. 

\medskip 
\noindent  $\bullet$
Let $I_{85}$ be the ideal generated by $2 \times 2$-minors of $A(F)$ where
\[ F = y^2z+ wzy-2w^2y+yz^2-9w^2z+3z^3+6w^3. \] 
The intial ideal of $I_{85}$ with respect to the lex order is a monomial
ideal whose saturation is $J_{85}$. 

\medskip
\noindent $\bullet$
Finally let $I_{102}$ be the ideal generated by $2 \times 2$-minors 
of $A(F)$ where
\[ F = y^2z+yz^2 + zw^2. \]
The intial ideal of $I_{102}$ with respect to the monomial order
$\prec_{[10,3,2,1]}$  is an ideal 
whose saturation is $J_{102}$.

In conclusion all the ideals in the list \ref{ExLigListe} are on the component
$H(3,1)$.  
\end{example}

\subsection{Ideals specializing to non-segment ideals}
To construct ideals like $I_{83}$, $I_{85}$ and $I_{102}$ in the example
above, we let
\[ F = \sum_{i=0}^{l-1} y^{l-1-i} F_{m+i} \] where
\[ F_{m+i} = \sum C_{i, \alpha_1,\alpha_2} z^{\alpha_1}w^{\alpha_2} \]
is the general form in $z$ and $w$ 
of degree $m+i$, and the $C_{i,\alpha}$ are variables.
Denote by $L$ the list of the
three $2\times 2$-minors of the matrix $A(F)$ and $I$ the ideal
generated by these minors.
Now fix a term order $\prec$, usually the lexicographic order, and
let $J$ be a Borel ideal. We want to assign values to the $C_{i,\alpha}$
such that the initial ideal of $I$ with respect to $\prec$ is an
ideal $\hat{J}$ whose saturation is $J$. We apply a Buchberger-like algorithm
as follows. We compute the $S$-polynomial of elements in $L$ and reduce
to a polynomial $h$. Let $m$ be the leading term in $h$ and $q(c)$ its 
coefficient, a polynomial in the $C$-variables.

\begin{itemize}
\item[1.] If $m$ is in $J$ we add a variable $c_m$ and a relation 
$c_m q(c) - 1=0$. 
\item[2.] If $m$ is not in $J$ we add the relation $q(c) = 0$, where
$q(c)$ is the coefficient of $m$. Find the 
largest term in $h$ after $m$. Let this be the new value of $m$.
Continue with 2. until $m$ is in $J$, and then go to 1.
\item [3.] Let  $L := L \cup \{h\}$, and continue with 1. or 2. after 
computing a new $S$-polynomial.
\end{itemize}

In the end we get a system of equations in the $C$-variables. If we can find
a solution to these we get an ideal $I$ whose initial ideal will have $J$
as its saturation. {If the system has no solutions, we try 
again fixing a new term order,   chosen so that 1. is used more often than 2. }
This procedure may not always succeed but in all cases
we have used it, it does.

This was the procedure that enabled us to produce $I_{83}, I_{85}$ and $I_{102}$
in Example \ref{ExExSyv} {and most of the explicit \Gr\ deformations in next example.}

\begin{example}
We consider the case The  $l=3$ and $m=3$, namely that of the ACM Borel ideal 
$J(3,3)=(x^2, xy^3, y^6)$, whose  Hilbert polynomial is  $9t-12$. By the applet
{\texttt BorelGenerator} we see that   there are $989$ Borel ideals on 
$\hilb^3_{9t-12}$ and that $J(3,3)$ is  is  $J_{989}$.
Among the remaining $988$ ideals, only $45$   fulfil the condition C1,  and all of 
these, save $J_{834}=(x^2, xy^3, xy^2z^4, xyz^5, xz^6, y^9)$,
fulfil the condition C2.

There are $28$ of these $44$ ideals that are reg-segment with respect to some 
term order. All these can be obtained as (the saturation of) the initial ideal of 
an ACM ideal $I$ of 
the type \ref{ExLigMatrise} for any general $F$.

   They are:
$ 768,  , 769,  770,$  $772,  775,  780$,  $788,  801,  817$, $ 875,  877,   880$, $ 887,  898,  913$, $927$,  $928$,   $938$, $939$,  $941$,  $954$,  $955, 
957$, $ 960, 977,  979,  981,$ $  984$.

\medskip

The other $16$ ideals are not reg-segment ideals.
 However,  they are a \Gr  \ degeneration of the ideal $I$ of the type
 \ref{ExLigMatrise} obtained for a suitable choice of the polynomial $F$.
Now we list the 16 Borel ideals and for each of them the polynomial $F$ and the term order we used. The first $10$     can be obtained using the lexicographic term order.

\begin{lista}
\item$J_{773}   =  (x^2, x y^3, x y^2 z, x y z^2, y^9, x z^{12} )$, $F=y^2z^3-w^3z^2+z^5+2wz^3y+w^2zy^2$ 
\item $ J_{776 }  =  (x^2, x y^3, x y^2 z, x y z^3, y^9, x z^{11} )$, $F=y^2 z^3+z^3 w^2+5 z^4 y+25 z^5+2 z^3 w y-6 w^2 z y^2$ 
\item $J_{781 }  =  (x^2, x y^3, x y^2 z, x y z^4, y^9, x z^{10} ) $, $F=y^2 z^3-w z^4-z^5-y z^4-w^2 z y^2$ 
\item  $ J_{783 }  =  (x^2, x y^3, x y^2 z^2, x y z^3, y^9, x z^{10} )$, $F=y^2  z^3+4  y  z^4+11  w  z^4-\frac{1127}{64}  w^3  z^2+3  z^5+7  w  z^3  y$ 
\item  $ J_{787 }  =  (x^2, x y^2, x y z^6, y^9, x z^9 )$, $F=y^2  z^3-2  w  z^3  y+5  w^2  z^2  y+z^5-\frac{85}{9}  w^2  z  y^2+\frac{25}{3}  w^3  y^2+3  w  z^4$ 
\item  $ J_{790 }  =  (x^2, x y^3, x y^2 z^2, x y z^4, y^9, x z^9 )$, $F=y^2 z^3+2 w z^3 y-z^5+2 w z^2 y^2-2 y z^4$ 
\item  $ J_{798 }  =  (x^2, x y^2, y^9, x y z^7, x z^8 )$, $F=9 y^2 z^3-18 w z^3 y+45 w^2 z^2 y+9 z^5-85 w^2 z y^2+75 w^3 y^2-11 y z^4+27 w z^4$ 
\item  $ J_{799 }  =  (x^2, x y^3, x y^2 z, x y z^6, y^9, x z^8 )$, $F=y^2 z^3-w  z^2  y^2+2  w  z^3  y  \sqrt{7}+4  w^2  z  y^2+7  z^5$ 

\item $ J_{804 }  =  (x^2, x y^3, x y^2 z^3, x y z^4, y^9, x z^8 )$, $F=y^2 z^3+8 z^5-3 w z^3 y+2 w z^4-12 y z^4$ 
\item  $ J_{814 }  =  (x^2, x y^3, x y^2 z^2, x y z^6, x z^7, y^9 ) $     , $F=y^2 z^3+6 z^5+6 w z^2 y^2+6 w z^3 y+6 w z^4+6 yz^4$        \end{lista}
For the last six we used a different term order given by a matrix  of the type:
\[M(v) = \left(\begin{array}{cccc} 1&1&1&1 \\ v_1& v_2& v_3 & v_4 
\\ v_5& v_6& v_7 & v_8  \\ 0&0& 1&0\end{array}\right).\]        
\begin{lista}
\item  $ J_{888 }  =  (x^2, x y^3, x y^2 z, x y z^2, y^9, y^8 z^5 )$, $F=y^2  z^3+w  z^4-2  z^5+w^2  z  y^2+w^3  z^2+w^2  z^2  y$, 

\noindent $v=[14, 2, 0, 0], [0, 0, 2, 1]$ 

\item  $ J_{899 }  =  (x^2, x y^3, x y^2 z, x y z^3, y^9, y^8 z^4 )$, $F=y^2  z^3-2  z^5+w^2  z  y^2+w^2  z^2  y+w  z^4$, 

\noindent $v=[14,2,0,0],[0,0,2,1]$

\item  $ J_{914 }  =  (x^2, x y^3, x y^2 z, x y z^4, y^9, y^8 z^3 )$, $F=y^2  z^3-2  z^5+w^2  z  y^2-y  z^4+w  z^4$, 

\noindent $v=[14, 2, 0, 0], [0, 0, 2, 1] $

\item  $ J_{930 }  =  (x^2, x y^3, x y^2 z^2, x y z^4, y^9, y^8 z^2 )$, $F=y^2  z^3-2  z^5+w  z^4-w^2  z^2  y$, 

\noindent $v= [14, 2, 0, 0], [0, 0, 1, -1]$ 

\item  $ J_{944 }  =  (x^2, x y^3, x y^2 z^3, x y z^4, y^9, y^8 z )$, $F=y^2  z^3+y  z^4-2  z^5+w  z^4$, 

\noindent $v= [14, 2, 0, 0], [0, 0, 7, 1]$  

\item  $ J_{978 } =  (x^2, x y^2, y^9, y^8 z, y^7 z^2 )$, $F=y^2  z^3+w^2  z^2  y-w^3  z  y+w  z^2  y^2+w^4  z-3  w^5$, 

\noindent $v=[12,2,0,0],[0,0,7,1] $  
\end{lista}
\end{example}

\section{Borel ideals on components of ACM curves}

We now consider ACM curves in $\ptre$ and investigate Borel ideals on
the Hilbert scheme component of such curves. We give various  sufficient
conditions for a Borel ideal to be on such a component.

The idea is to
construct families of ACM curves and find Borel ideals which are
specializations of such families. This is intricate and we shall focus
here on the case that the curves are on a quadric, that is the case
$a = 2$ in \ref{NecLigJ}. Related to this is \cite{Le} by P.Lella
where he constructs families of ideal parametrized by a rational curve,
connecting two Borel ideals. In this way he can proceed stepwise between
Borel ideals. However as soon as one has used two steps or more, one cannot
be sure that the starting and ending Borel ideal is on the same
component.

\subsection{Conditions on Borel ideals}
Let $J$ be a saturated Borel ideal on the component $H(l,m)$ of 
$J(l,m) = (x^2, xy^l, y^{l+m})$. By Theorem \ref{NecTheMain}
all the monomials $x^2, xy^l$ and $y^{2l+m}$ are in the ideal $J$. Hence
$J$ must have the following set of generators for some $0 \leq p \leq l$. 
\begin{eqnarray} 
\notag  & & x^2, \\
\label{BorelLigab} & &  xy^{l-p}z^{a_0}, xy^{l-p+1}z^{a_1},
\ldots, xy^{l-1}z^{a_{p-1}}, xy^l, \\ 
\notag & & y^{l+m+p}z^{b_p}, \ldots, y^{2l+m-1}z^{b_{l-1}}, y^{2l+m},
\end{eqnarray}
where the sequences $a_0,a_1, \ldots, a_{p-1} $ and $b_p, b_{p+1}, 
\ldots, b_{l-1}$ are
strictly decreasing until they possibly become zero.  

It will also be convenient to allow these sequences to be weakly 
decreasing. We call such an ideal an {\it almost Borel ideal}.

\begin{lemma} The ideal $\Bo(l,m)$ and the ideal $J$ with generators
\ref{BorelLigab} have the same Hilbert polynomial iff
$ \sum a_i + \sum b_i = \sum_{i = 0}^{p-1} (m+2i)$.
\end{lemma}

\begin{proof}
Suppose $a_j > a_{j+1}$. Let $J^\prime$ be the ideal with the generator
$xy^{l-j} z^{a_j}$ replaced by $xy^{l-j}z^{a_j - 1}$. There is then an
exact sequence
\[ 0 \pil (xy^{l-p+j}z^{a_j-1}) \pil S/J \pil S/J^\prime \pil 0 \]
where the first submodule consists of elements $xy^{l-p+j}z^{a_j-1}w^r$
for $r \geq 0$. 
Hence the difference of the Hilbert polynomials of the two latter
quotient rings is just $1$. The same thing happens when we
reduce some $b_k$. 

Now let $J_{p,q}$ be the ideal generated by 
$x^2, xy^{l-p}, y^{l+m+q}$. By successively reducing the $a$'s and $b$'s
we find for the Hilbert polynomials
\begin{equation} \label{BorelLigJI}
HP_{J_{p,p}}(d) =  HP_{J}(d) + \sum a_i + \sum b_i.
\end{equation}
Now there is an exact sequence
\[ 0 \pil (xy^{l-p-1}) \pil S/J_{p,q} \pil S/J_{p+1,q} \pil 0 \]
and the kernel is $xy^{l-p-1}\langle z,w \rangle^{d-(l-p)}$ in degree $d$ and so 
has Hilbert polynomial $d-(l-p) + 1$. 
There is also an exact sequence
\[ 0 \pil (y^{l+m+q}) \pil S/J_{p,q+1} \pil S/J_{p,q} \pil 0. \]
The kernel has Hilbert polynomial $d-(l+m+q) + 1$. 
From this we readily get that the Hilbert polynomial
$HP_{J_{p+1,p+1}}(d)$ is equal to $HP_{J_{p,p}}(d) + (m+2p)$.
Starting with $J_{0,0} = \Bo(l,m)$ we therefore get that 
\[ HP_{J_{p,p}}(d) = HP_{\Bo(l,m)}(d) + \sum_{i = 0}^{p-1} (m+2i).\]
Comparing this with \ref{BorelLigJI}, 
$\Bo(l,m)$ and $J$ have the same Hilbert function iff the numerical
equality holds.
\end{proof}

\subsection{Main theorems of sufficiency}
For Borel ideals with $p \leq 2$ we now give sufficient conditions
for a Borel ideal to be on the component $H(l,m)$. 

\begin{theorem}[$p = 1$] \label{BorelThep1}
When $\bfa = (m-i)$ and $\bfb = 
(i,0,\ldots, 0)$, the Borel ideal \ref{BorelLigab} is on 
the component $H(l,m)$. 
\end{theorem}

\begin{theorem} [$p=2, \bfb = \bfo$] \label{BorelThep2spes}
When $\bfa = (m+2+i,m-i)$ and
$\bfb = \bfo$, the Borel ideal \ref{BorelLigab} is on 
the component $H(l,m)$. 
\end{theorem}

More generally we can show:

\begin{theorem}[$p = 2$] \label{BorelThep2} When $\bfa = (a_0, a_1)$ and 
$\bfb = (b_2, 0, \ldots, 0)$ with $a_0 + a_1 + b_2 = 2m+2$, 
the Borel ideal \ref{BorelLigab}
is on the component $H(l,m)$ if either: 
\begin{itemize}
\item[1.] $ a_0 \geq m+2$, or
\item[2.] $ a_0 \leq m+2$ and $a_0 - a_1$ is even.
\end{itemize}
\end{theorem}

The proofs of these theorems is developed in the next section.

We also have a more general result.
Given non-negative integers with 
\[ p_0 \leq p_1 +1 \leq p_2 + 2 \leq \cdots \leq p_{l-1} + l-1, \]
and assume $p_0$ is the minimum of the $p_i$'s.
Consider partitions
\[ \lambda : \la_1 \geq \la_2 \geq \cdots \geq \la_r (\geq 0) \] 
consisting of $r$ parts of sizes $\leq l-1$. Let $p_\la
= \sum p_{\la_i}$.

\begin{theorem} \label{BorelThePart} With notation and assumptions as
above assume also for each $r = 0, \ldots, l-1$ that 
$r p_{l-r} \geq p_\la$ for all partitions $\la$ of $r(l-r)$ 
into $r$ parts of sizes $\leq l-1$. (In other words $p_\la$ achieves
its maximum when all parts are equal.) 
Then the Borel ideal \ref{BorelLigab} with $p = l$ and 
\[ a_i = m+2(l-1-i) + (l-1-i) p_{i+1} - (l-i)p_i + p_0 \]
for $i = 0, \ldots, l-1$, is on the component $H(l,m)$.
\end{theorem}

In particular letting each $p_i = 0$ we see that the ideal
with $a_i = m + 2(l-1-i)$ for $i = 0, \ldots, l-1$ 
is on the component $H(l,m)$.

\subsection{Auxiliary results}
In the end we now give some auxiliary results which will be repeatedly
used in our arguments for the above results.
When $I$ is a monomial ideal in $k[x,y,z,w]$ we may make a coordinate
change $w \pil w + \lambda z$ and let $J$ be the initial ideal
for any monomial order where $z > w$.
Note that if {$x^ay^bz^{c_1}w^{c_2}$ is in $I$, then 
$x^a y^b z^{c_1 + c_2}$ is in $J$. The following is clear.

\begin{lemma} Any component of the Hilbert scheme containing $I$
will also contain $J$. 
\end{lemma}

We call $J$ the {\it $z$-transform} of $I$. The following will
be used frequently.

\begin{satlemma} \label{BorelLemIJK}
Let $I$ be a monomial ideal with the same Hilbert polynomial as $\Bo(l,m)$.
If the saturation of $I$ (resp. of the $z$-transform of $I$) contains 
an {almost Borel ideal} $K$ of the form \ref{BorelLigab} with  
\[ \sum a_i + \sum b_i = \sum_{i = 0}^{p-1} (m+2i), \] 
then $K$ is the saturation of $I$ (resp. the $z$-transform of $I$).
\end{satlemma}

\begin{proof}
Clearly $I$ and $K$ have the same Hilbert polynomial. So $K/I$ is of finite
length. Since $K$ is saturated, it must be the saturation of $I$.
The argument for the $z$-transform is similar.
\end{proof}

\section{Equations of families of ACM curves on a quadric}

We will now describe explicitly the families of ACM curves that we
shall work with and whose degenerations will be Borel ideals.

\subsection{The family of ideals}
Denote by $R$ a polynomial ring $\kk[s,\{t_k\}]$. 
The associated affine space will be a parameter space for the family of ideals $\tI$
we shall work with.
This is the family generated 
by the $2 \times 2$ minors of the matrix 
\[ A_R(F) = \left [ \begin{matrix} x & sy^m & -F \\
                          0  & x & y^l
           \end{matrix}   \right ]  \]
where $F$ is a polynomial in $R[x, y, z, w]$, homogeneous
of degree $l+m-1$ in the $x,y,z,w$. 
The $2 \times 2$ minors of the matrix are 
\[ x^2,\quad  x y^l, \quad
 G := xF + sy^{m+l}. \]

Performing row and column operations
on the matrix, we may  assume 
\[ F = y^{l-1}F_{l-1} + y^{l-2} F_{l-2} + \cdots + y^{l-q}F_{l-q}. \]
Of course the most general is having $q = l$. We write it however
in this way since $q$ will be a natural parameter in the families we
construct. By considering $y^q G$ note that $sy^{m+l+q}$ is in 
$\tI$.


We now assume that each $F_{l-i}$ is the product of a monomial
in $z,w$ and a variable in the ring $R$. More specifically 
we assume that $F$ has the following form
\[ t_0y^{l-1}z^{m-p_{l-1}}w^{p_{l-1}} + 
t_1y^{l-2}z^{m+1-p_{l-2}}w^{p_{l-2}} + 
\cdots + t_{q-1}y^{l-q}z^{m+q-1 - p_{l-q}}
w^{p_{l-q}}. \]
Then $\tI$ is generated by $x^2$, $xy^l$, and 
\[ G_0 = G =  t_0xy^{l-1}z^{m-p_{l-1}}w^{p_{l-1}} + \cdots 
+ t_{q-1}xy^{l-q}z^{m+q-1 - p_{l-q}}w^{p_{l-q}} + sy^{m+l}. \]

Via a ring homomorphism $R = \kk[s,t_0, \ldots, t_{l-q}] \pil \kk[t, t^{-1}]$,
the family of ideals $\tI$ maps to 
a family of ideals $\tI_t$ parametrized by a rational parameter $t$. This 
new family has a limit ideal when $t \pil \infty$. We shall consider
monomial maps where $s \mapsto t^w$ and $t_i \mapsto t^{w_i}$.
When  $w$ and the $w_i$'s are sufficiently general this limit will be
a monomial ideal. 

 Now let $R$ have a monomial order. Given an integer $N$, by D.Bayer's
thesis, \cite{Ba}, there are
integer weights $w$ and $w_i$'s such that for two monomials of degrees 
$\leq N$, 
\[ m_\bfa = s^at_0^{a_i}\cdots t_{l-q}^{a_{l-q}}, \quad
  m_\bfb = s^bt_0^{b_i}\cdots t_{l-q}^{b_{l-q}}, \] 
we have $m_\bfa > m_\bfb$ iff the scalar products
$wa + \sum w_i a_i > wb + \sum w_i b_i$. 
For this set of weights associate the map $R \pil \kk[t, t^{-1}]$ given by 
$s \mapsto t^w$ and 
$t_i \mapsto t^{w_i}$. 
We then get a monomial ideal as the limit of $\tI_t$ 
 when $t \pil \infty$.} This
limit ideal depends only on the family $\tI$ 
and the monomial ordering on $R$ and not
on the choice of weights $w$ and $w_i$'s.

We shall in the next sections
consider various monomial orders on $R$ and find the limit
ideals associated to these orders.

\medskip

The ideal $\tI$ contains $yG_0$. Since $xy^l$ is in $\tI$, this
is congruent modulo $\tI$ to
\[ G_1 = t_1xy^{l-1}z^{m+1-p_{l-2}}w^{p_{l-2}} + \cdots 
+ t_{q-1}xy^{l-q+1}z^{m+q-1 - p_{l-q}}w^{p_{l-q}} + sy^{m+l+1}. \]
More generally $y^r G_0$ is congruent modulo $\tI$ to 
\[ G_r =  t_rxy^{l-1}z^{m+r-p_{l-1-r}}w^{p_{l-1-r}} + \cdots 
+ t_{q-1}xy^{l-q+r}z^{m+q-1 - p_{l-q}}w^{p_{l-q}} + sy^{m+l+r}. \]
Note that all the polynomials $G_0, G_1, \ldots, G_{l-q}$ are in $\tI$.

\subsection{Equations and limits when $q = 2$}
We let $q = 2$, $p_1 = i$ and $p_0 = 0$. In this case we get
\begin{eqnarray*}
G_0 & = & t_0xy^{l-1}z^{m-i}w^i + t_1 xy^{l-2}z^{m+1} + s y^{l+m} \\
G_1 & = & t_1xy^{l-1}z^{m+1} + s y^{l+m+1}.
\end{eqnarray*}

In order to eliminate $xy^{l-1}$ from these equations we let
\begin{eqnarray*}
G_{01} &=& t_1z^{i+1} G_0 - t_0w^i G_1 \\
       &=& t^{2}_1xy^{l-2}z^{m+2+i} + st_1y^{m+2}z^{i+1} 
- st_0y^{l+m+1}w^i.
\end{eqnarray*}

The following shows Theorem \ref{BorelThep1} and Theorem \ref{BorelThep2spes}.

\begin{proposition} Let $t_0 > t_1 > s$.  \label{EqProq2}
\begin{itemize}
\item[1.] When $t_1^2 > st_0$ the saturation of the limit ideal
is Borel with $\bfa = (m+2+i, m-i)$ and $\bfb = \bfo$. 
\item[2.] When $st_0 > t_1^2$ the $z$-transform
of the saturation of the limit ideal, has saturation a 
Borel with $\bfa = (m-i)$ and $\bfb = (i,0,\ldots,0)$.
\end{itemize}
\end{proposition}

\begin{proof} By $G_1$ the limit contains $xy^{l-1}z^{m+1}$, and
by $G_0$ the limit contains $xy^{l-1}z^{m-i}w^i$. Hence the
saturation of the limit contains $xy^{l-1}z^{m-i}$.
Also considering $y^2G_0$ we see that the limit contains $y^{m+l+2}$.

In case 1. the limit contains $xy^{l-2}z^{m+2+i}$ by 
$G_{01}$.
Hence the saturation of the limit contains
\[ xy^{l-2}z^{m+2+i}, xy^{l-1}z^{m-i}, y^{l+m+2}, \]
and by the Saturation Lemma \ref{BorelLemIJK} applied to the appropriate ideal
$K$ as in \ref{BorelLigab} with $p=2$, we get part 1.

In case 2. the limit contains $y^{l+m+1}w^i$ by $G_{01}$. 
So the $z$-transform of the saturation of the limit contains
\[ xy^{l-1}z^{m-i}, y^{l+m+1}z^i, y^{l+m+2}. \]
Hence by the again by the Saturation Lemma \ref{BorelLemIJK} 
applied to the appropriate ideal
$K$ as in \ref{BorelLigab} with $p=1$, we get part 1.
\end{proof}

\subsection{The equations when $q = 3$.}

We let $p_2 = i$,  $p_1 = j$, and $p_0 = 0$. 
Also let $\Delta = i - 2j$, the second difference of the $p$'s. 
The ideal $\tI$ contains:
\begin{eqnarray*}
G_0 &=& t_0 xy^{l-1}z^{m-i}w^i + t_1 xy^{l-2}z^{m+1-j}w^j + t_2xy^{l-3}z^{m+2}
+ sy^{l+m} \\
G_1 & = & t_1xy^{l-1}z^{m+1-j}w^j + t_2 xy^{l-2}z^{m+2} + sy^{l+m+1} \\
G_2 & = & t_2xy^{l-1}z^{m+2} + sy^{l+m+2}.
\end{eqnarray*}

We now eliminate $xy^{l-1}$ from $G_0$ and $G_1$. When $i \geq j$ we let
\begin{eqnarray*}
G_{01} &=& t_1 z^{1+i-j}G_0 - t_0w^{i-j}G_1\\
  &=& t_1^2xy^{l-2}z^{m+2+\D}w^j + t_1t_2xy^{l-3}z^{m+3+\D + j} + st_1  
y^{l+m} z^{1+\D + j}\\
 &-& t_0t_2xy^{l-2}z^{m+2}w^{\D + j} - st_0y^{l+m+1}w^{\D + j} .
\end{eqnarray*}

We then obtain when $i \geq j$ that
\begin{eqnarray} \notag
xy^{l-2}z^{m+2+\D}w^j &\equiv& \frac{t_0t_2}{t_1^2}xy^{l-2}z^{m+2}w^{j+\D} - 
\frac{t_2}{t_1}x^{l-3}z^{m+3+\D+j}
- \frac{s}{t_1}y^{m+l}z^{1+\D + j}  \\ \label{Limq3LigBgG01}
&+& \frac{st_0}{t_1^2}y^{m+l+1}w^{\D+j} \pmod {\tI}
\end{eqnarray}

We eliminate $xy^{l-1}$ from $G_1$ and $G_2$ by letting
\begin{eqnarray*}
G_{12} & = & t_2z^{1+j}G_1 - t_1w^jG_2 \\
 & = & t_2^2 xy^{l-2}z^{m+3+j}  + st_2 y^{l+m+1} z^{1+j} - 
st_1y^{l+m+2} w^j.
\end{eqnarray*}

From this we obtain
\begin{equation} \label{Limq3LigBgG12}
xy^{l-2}z^{m+3+j} \equiv \frac{st_1}{t_2^2} y^{m+l+2}w^j 
- \frac{s}{t_2}y^{m+l+1}z^{1+j} 
\pmod {\tI}
\end{equation}

\medskip
Now we want to eliminate $xy^{l-1}$ and $xy^{l-2}$ from the equations
of $G_0, G_1$, and $G_2$. By taking $2 \times 2$-minors of the 
coefficients of these monomials we let 
\[ G_{012} = t_2^2z^{2m+4} G_0 - t_1t_2z^{2m+3-j}w^j G_1
+(t_1^2z^{2m+2-2j}w^{2j} - t_0t_2z^{2m+2-i} w^i)G_2.\]

When $\Delta \leq 0$, which is equivalent to $i \leq 2j$ 
we can factor out $z^{2m+2-2j}$ from the coefficients and let
\begin{eqnarray*} G_{012}^- &=& G_{012}/z^{2m+2-2j} \\
&=& t_2^3xy^{l-3}z^{m+4+2j} + st_2^2y^{l+m}z^{2j+2} - 
st_1t_2 y^{l+m+1}z^{1+j}w^j \\
&+&st_1^2y^{l+m+2}w^{2j} - st_0t_2
y^{l+m+2} z^{2j-i}w^i.
\end{eqnarray*}

\subsection{Proof of Theorem \ref{BorelThep2}}
Until now we have assumed monomials containg
$s$ always to be smaller than those without $s$. Now we drop this
asumption but continue to assume $t_0 > t_1 > t_2$. 
Our aim is now to prove Theorem \ref{BorelThep2}.

\begin{proposition} \label{Limq3ProP2A} Let $i \geq j$. If  $t_2 > s$, and 
$st_1 > t_2^2$,
the saturation of the limit ideal has $z$-transform whose saturation is 
the Borel ideal with $\bfa = (m+2+i-j, m-i)$ and $\bfb = (j)$.
\end{proposition}

This proves Part 1. of Theorem \ref{BorelThep2}.

\begin{proof}
Assume first  $t_1^2 > t_0t_2$. The following monomials are in  the limit.
\begin{itemize}
\item [1.] $xy^{l-2}z^{m+2+\D}w^j$ by \ref{Limq3LigBgG01}, since
$t_1^2 > t_0t_2 > t_0s$.
\item [2.] $xy^{l-1}z^{m-i}w^j$ by considering $G_0$.
\item [3.] $xy^{l-1}z^{m+2}$ by considering $G_2$.
\item [4.] $y^{m+l+2}w^j$ by \ref{Limq3LigBgG12}.
\end{itemize}

By 2. and 3.,$xy^{l-1}z^{m-i}$ is in the saturation of the limit.
The $z$-transform of this ideal then contains
\[ xy^{l-2}z^{m+2+i-j}, xy^{l-1}z^{m-i}, y^{m+l+2}z^j \]
and so by the Saturation Lemma \ref{BorelLemIJK},
we obtain the statement in this case.

\medskip
When $t_0t_2 > t_1^2$ we still have $t_0 t_2 > t_0s$. 
The only change to the above is that by
\ref{Limq3LigBgG01} the limit contains $xy^{l-2}z^{m+2}w^{\D +j}$ instead
of $xy^{l-2}z^{m+2+\D}w^j$, but both have the same $z$-transform, giving
the result.
\end{proof}

\begin{proposition} \label{Limq3ProP2B}
Let $j \leq i \leq 2j$. Assume
\[ t_1t_2^2 > st_1^2 > t_2^3,  \quad t_1^2 > t_0t_2.  \]
The $z$-transform of the limit ideal
has as saturation the Borel ideal with $\bfa = (m+2+i-2j, m-i)$ and
$\bfb = (2j)$.
\end{proposition}

This proves Part 2. of Theorem \ref{BorelThep2}

\begin{proof} Note that $t_2^2 > st_1$ implies $t_2 > s$.
The following monomials are in the saturation of the limit.
\begin{itemize}
\item [1.] $xy^{l-1}z^{m-i}$ by considering $G_0$ and $G_2$.
\item [2.] $xy^{l-2}z^{m+2+\D}w^{j}$ by \ref{Limq3LigBgG01}.
\item [3.] $xy^{l-2}z^{m+3+j}$ by \ref{Limq3LigBgG12}.
\item [4.] $y^{m+l+2}w^{2j}$ by $G_{012}^-$ since $st_1^2$ is greater than
$st_0t_2$ and $t_2^3$.
\end{itemize}

Parts 2. and 3. above give that $xy^{l-2}z^{m+2+\D}$ is in the saturation.
Taking the $z$-transform of the ideal, we get the statement.
\end{proof}

\section{A family of limit ideals when $q = l$}

In this section we prove Theorem \ref{BorelThePart}.
We assume that
\begin{equation}
p_0 \leq p_1 +1 \leq p_2 + 2 \leq \cdots \leq p_{l-1} + l-1, 
\label{FamLabps} 
\end{equation}
\begin{equation}
\label{FamLabpmin}
p_0 = \min \{ p_i \, | \, i + 0, \ldots, l-1\}, 
\end{equation}
\begin{equation}
\label{FamLabt}
t_0 > t_1 > \cdots > t_{l-1} > s. 
\end{equation}
When 
\[ \lambda : \la_1 \geq \la_2 \geq \cdots \geq \la_k (\geq 0) \] 
is a partition
consisting of $k$ parts of sizes $\leq l-1$, we let $p_\la
= \sum p_{\la_i}$ and the monomial $t_\la = \prod t_{\la_i}$. We also let 
$\overline{\la}$ be the complementary partition given by 
$\overline{\la}_i = l-1-\la_{k+1-i}$ for $i = 1, \ldots, k$. 

Let us state the result a bit more elaborately.

\begin{theorem}
Assume that in addition to \ref{FamLabps}, \ref{FamLabpmin}, and
\ref{FamLabt}, for each $r = 0, \ldots, l-1$ 
\begin{itemize}
\item [1.] $r p_{l-r} \geq p_\la$ for all partitions $\la$ of $r(l-r)$ 
into $r$ parts of sizes $\leq l-1$. 
\item [2.] $t_r^{r+1} > t_\lambda$ for all partitions $\la$ of
$r(r+1)$ into $r+1$ parts of sizes $\leq l-1$. 
\item [3.] All monomials containing $s$ are smaller than monomials
without $s$. 
\end{itemize}
Then the limit ideal has as saturation the Borel ideal with  $p = l-1$
and 
\[ a_i = m+2(l-1-i) + (l-1-i) p_{i+1} - (l-i)p_i + p_0. \]
\end{theorem}

\begin{proof}
Let $G = [G_0, G_1, \ldots, G_{l-1}]^t$ and 
$Y = [xy^{l-1}, xy^{l-2}, \ldots, x]^t$. 
Then $G = A Y + s E$
where $E = [y^{l+m}, y^{l+m+1}, \ldots, y^{2l+m-1}]^t$, and
$A$ is an $l\times l$ (symmetric) matrix with rows and columns indexed by 
$0, \ldots, l-1$, and the entry in position  $(i,j)$ is 
\[ t_{i+j} z^{m+i+j-p_{l-1-i-j}}w^{p_{l-1-i-j}}. \]
Let $A(r)$ for $1 \leq r \leq l-1$ 
be the submatrix of $A$ consisting of the first $r+1$ rows and $r$
columns, and let $M_i$ be the maximal minor of $A(r)$ obtained by 
omitting row $i$. We may eliminate  $xy^{l-1}, \ldots, xy^{l-r}$ from 
$G_0, G_1, \ldots, G_r$ by forming
\[ G_{01..r} = M_0 G_0 - M_1 G_1 + M_2 G_2 + \cdots + (-1)^r M_r G_r. \]
The minor $M_0$ will be an alternating sum of terms of the form
\[ t_{\overline{\la}} z^{r(m+r) - p_\la} w^{p_\la} \]
where the $\la$ are partitions of $r(l-r-1)$ into $r$ parts of sizes 
$\leq l-1$. 
In particular note that the term
\[ t_r^r z^{r(m+r) - r p_{l-r-1}} w^{rp_{l-r-1}} \]
occurs as the product of the elements on the anti-diagonals in the submatrix
of $A(r)$ we get by omitting the first row.  The terms of $M_i$ will be
a sum of terms 
\[ t_{\overline{\la}} z^{r(m+r) - i- p_\la} w^{p_\la} \]
where $\la$ is a partition of $r(l-r-1) + i$ into $r$ parts of sizes $\leq
l-1$. Note that if $i < r$ we can in each term above increase a suitable
part of $\la$ to obtain a partition $\la^\prime$ of $r(l-r-1) + i+ 1$.
Then 
\[ r(m+r) - i- p_\la \geq r(m+r) - i- 1 - p_{\la^\prime}. \]
So the lowest power of $z$ will occur in the minor $M_r$. By 
assumption 1. the term with the lowest power in $z$ is  
\[ t_{r-1}^r z^{r(m+r-1) - r p_{l-r}} w^{r p_{l-r}}. \]
We may then divide each $M_i$ by 
the product $z^{r(m+r-1) - r p_{l-r}}$
and let $M_i^\prime$ be the quotient.
Now form  
\[ G_{01..r}^\prime = M_0^\prime G_0 - M_1^\prime G_1 + \ldots . \]
The question is now what is the limit of $G_{01..r}^\prime$. 
In $M_i^\prime$ all $t$-monomials are of the form $t_{\overline{\la}}$
where $\lambda$ is a partition of $r(l-r-1) + i$ into $r$ parts
of sizes $\leq l-1$. Recall that $G_i$ consists of terms
\[ t_{i+j}xy^{l-j-1}z^{m+i+j-p_{l-1-i-j}}w^{p_{l-1-i-j}}. \]All the terms containing
$xy^{l-1}, \ldots, xy^{l-r}$ are eliminated in $G_{01..r}^\prime$.
Therefore the largest term in $M_i^\prime G_i$ which does not get 
eliminated will occur in 
\begin{equation} \label{LimqlLigMt}
 M_i^\prime \cdot  t_{r+i} xy^{l-r-1} z^{m+r+i - p_{l-1-r-i}} w^{p_{l-1-r-i}}.
\end{equation}

Now $\overline{\la}$ is a partition of $r^2 - i$ into $r$ parts of
sizes $\leq l-1$. 
Then $\overline{\la}, r+i$ will be a partition of $r(r+1)$
into $r+1$ parts of sizes $\leq l-1$. By assumption 2. the largest
monomial in the $t$'s among the products \ref{LimqlLigMt}
occurs when $i = 0$ and 
\[\overline{\la} : r,r,\ldots, r,\] so 
\[\la : l-1-r, l-1-r, \ldots, l-1-r. \] 
Hence the limit term of $G_{01..r}^\prime$ 
will be the term occuring in $M_0^\prime G_0$:
\begin{eqnarray*}
& &  z^{r+rp_{l-r} - rp_{l-r-1}}w^{rp_{l-r-1}} \cdot xy^{l-r-1}
z^{m+r-p_{l-r-1}} w^{p_{l-r-1}} \\
& = & xy^{l-r-1} z^{m+2r+rp_{l-r} - (r+1)p_{l-r-1}}w^{(r+1)p_{l-r-1}}. 
\label{FamLigLang}
\end{eqnarray*}

When $r=l-1$, by assumption \ref{FamLabpmin}, we may note that the
minimum power of $w$ that can occur in an $M_i$ for $i = 0, \ldots, 
l-1$, is $w^{(l-1)p_0}$ occuring in $M_{l-1}$.
Thus we may divide out by this monomial also in 
all the $M_i^\prime$ and then obtain a limit term
\[ xz^{m+2(l-1) + (l-1)p_1 - lp_0} w^{p_0}. \]
Now consider the monomial 
\begin{equation} \label{FamLabxyr}
xy^{l-r-1} z^{m+2r+rp_{l-r} - (r+1)p_{l-r-1}}w^{p_0}.
\end{equation}

\begin{claim}
The sequence $m+2r +rp_{l-r} - (r+1)p_{l-r-1}$ is
weakly increasing for $r = 0, \ldots, l-1$.
\end{claim}

\begin{proof}
Consider $r$ in the range $1, \ldots, l-1$, and the equation we want to prove:
\[ 2(r-1) + (r-1)p_{l-r+1} - rp_{l-r} \leq 2r + rp_{l-r} - (r+1)p_{l-r-1}.\] 
Using that $p_{l-r-1} \leq p_{l-r} + 1$, this is implied by
\[ (r-1)p_{l-r+1} + (r-1)p_{l-r-1} \leq (2r-2) p_{l-r}. \]
(Note that although $p_{l-r+1}$ is not defined when $r = 1$, the 
coefficient above is zero.) Dividing out by $r-1$,
this is again a consequence of
\[ p_{l-r+1} + (r-2) p_{l-r} + p_{l-r-1} \leq r p_{l-r} \]
which holds by assumption 1. in the theorem.
\end{proof}


The monomials \ref{FamLabxyr} are therefore all in the saturation 
of the limit ideal.
Taking the $z$-transforms they become
\[  xy^{l-r-1} z^{m+2r+rp_{l-r} - (r+1)p_{l-r-1} + p_0}.\]
Now the sum of all these powers of $z$ as $r = 0, \ldots, l-1$ 
telescopes to $\sum_{r = 0}^{l-1} m+2r$, so by the Saturation Lemma 
\ref{BorelLemIJK} we
obtain the statement of the theorem.
\end{proof}

{\section*{Acknowledgments} 
The third author was supported by the framework of PRIN 2010-11 
\emph{Geometria delle variet\`a algebriche}, cofinanced by MIUR. 
}

\bibliographystyle{amsplain}
\bibliography{BibliographyBorel}
\end{document}